\documentclass[12pt]{amsart}
\usepackage{amscd,amssymb,graphics,lineno}
\usepackage{mathrsfs} 

\newtheorem{theorem}{Theorem}[section]
 
\newtheorem{lemma}[theorem]{Lemma}

\theoremstyle{definition}

\theoremstyle{remark}

\numberwithin{equation}{section}

\newcommand{\abs}[1]{\lvert#1\rvert}

\def\Q {{\mathbb Q}}
\def\s{{\mathbb S}}

\def\N{{\mathbb N}}

\def\R{{\mathbb R}}

\def\e{\varepsilon}

\def\Aut{{\mbox{\rm Aut}\,}}

\def\UCB{{\mbox{\rm UCB}\,}}

\def\CB{{\mbox{\rm CB}\,}}
\def\St{{\mbox{\rm St}}}

\newcounter{quest}
\stepcounter{quest}

\begin{document}

        \title[Equivariant compactifications]
        {A topological transformation group without non-trivial equivariant compactifications
}

\author[V.G. Pestov]{Vladimir G. Pestov}

\address{Department of Mathematics and Statistics, 
University of Ottawa, 585 King Edward Ave., Ottawa, Ontario, Canada K1N 6N5}

\address{Departamento de Matem\'atica, Universidade Federal de Santa Catarina, Trindade, Florian\'opolis, SC, 88.040-900, Brazil}
\email{vpest283@uottawa.ca}




\begin{abstract} 
        There is a countable metrizable group acting continuously on the space of rationals in such a way that the only equivariant compactification of the space is a singleton. This is obtained by a recursive application of a construction due to Megrelishvili, which is a metric fan equipped with a certain group of homeomorphisms.
        The question of existence of a topological transformation group with the property in the title was asked by Yu.M. Smirnov in the 1980s. 
\end{abstract}

\maketitle

\section{Introduction}
Let a (Hausdorff) topological group $G$ act continuously on a (Tychonoff) topological space $X$. An {\em equivariant compactification} of $X$ is a compact space $K$, equipped with a continuous action by $G$, together with a continuous equivariant map $i\colon X\to K$ having a dense image in $K$. (See \cite{dV-2,dV-1,dV,stoyanov,megrelishvili,megrelishvili07,vM}.) The question of interest is whether, given $G$ and $X$, there exists a compactification into which $X$ embeds topologically, that is, $i$ is a homeomorphism onto its image.

For instance, this is so if a topological group $G$ acts on itself by left translations, in which case the compactification is the greatest ambit of $G$ (\cite{brook}; \cite{P06}, p. 42). If the acting group $G$ is discrete, then its action on $\beta X$, the Stone-\v Cech compactification of $X$, is clearly continuous, so $i$ is the canonical homeomorphic embedding  $X\hookrightarrow\beta X$. For non-discrete $G$, the action on the Stone-\v Cech compactification $\beta X$ is continuous just in some exceptional cases. Nevertheless, if the acting group $G$ is locally compact, then every $G$-space $X$ admits an equivariant compactification into which $X$ embeds homeomorphically, this is an important result by de Vries \cite{dV-1,dV}. (For compact Lie groups, the result was known to Palais much earlier, \cite{palais}.)

For a long time it was unknown whether the same conclusion holds in the case of every acting topological group (the question was advertised by de Vries in 1975, see \cite{dV-2}). However, Megrelishvili \cite{megrelishvili,megrelishvili89} has shown it is not so, by constructing an example where both $G$ and $X$ are Polish, yet the embedding $i$ is never topological for any equivariant compactification of $X$.
More examples can be found in the work of Megrelishvili and Scarr \cite{MegS} and Sokolovskaya \cite{sokolovskaya}. (The earliest claim of an example \cite{AS} was withdrawn by its authors.) In some examples of Megrelishvili, the mapping $i$ is not even injecive.

Some time in the 1980s, Yuri M. Smirnov asked about the existence of a topological group $G$ acting on a (non-trivial) space $X$ in such a way that the only equivariant compactification of $X$ is the singleton. The question was apparently never mentioned in Smirnov's papers; however, it was well known among the Moscow general and geometric topologists, and later documented in Megrelishvili's papers \cite{megrelishvili89,megrelishvili07}. (See also a discussion in \cite{P06}, Rem. 3.1.6.)
Here we notice that a topological transformation group conjectured by Smirnov indeed exists. 

The construction is based on the example of Megrelishvili \cite{megrelishvili,megrelishvili89}, the first ever in which the mapping $i\colon X\to K$ is not a topological embedding. In this construction, the space $X$ is the metric fan, joining the base point $\ast$ with countably many endpoints $x_n$ with the help of intervals of unit length, and equipped with the graph metric. The group $G$ is a certain subgroup of the group of homeomorphisms of $X$ with the compact-open topology, chosen in a certain subtle way.  Megrelishvili has shown that in every equivariant compactification of $X$, the images of the points $x_n$ converge to the image of the base point $\ast$, so the compactification mapping $i$ is not homeomorphic. A further observation by Megrelishvili is that if one joins two copies of such a fan by identifying their respective endpoints and extends the group action over the resulting space, then the image of the sequence of $(x_n)$ will converge to the image of each the two base points, thus these two points have a common image, and consequently $i$ is not even injective. 

The idea of our example is, starting with a topological transformation group, to attach to it a copy of the double Megrelishvili fan (in its rational version) at every pair of distinct points $x,y$. As a result, in any equivariant compactification of the resulting space, $X$ collapses to a point. If we repeat the same construction countably many times, the union of the increasing chain of resulting spaces has no nontrivial compactifications. 
The topological space that we obtain is countably infinite, metrizable, and has no isolated points, and so is homeomorphic to the space of rational numbers.


The article is concluded with a discussion of some open questions.

\section{Universal equivariant compactifications}

In this and the next section we will discuss known results, with a view to make the presentation self-sufficient.

All our actions will be on the left. 
If a topological group $G$ acts continuously on a compact space $K$, then the associated representation of $G$ by linear isometries on the space $C(K)$, 
\[{^gf}(x) = f(g^{-1}x),~~g\in G,~x\in K,~f\in C(K),\]
is continuous. Now let $G$ also act continuously on a Tychonoff space $X$, and let $i\colon X\to K$ be an equivariant continuous map with a dense image; that is to say, the $G$-space $K$ is an {\em equivariant compactification} of the $G$-space $X$. For every $f\in C(K)$, the composition $f\circ i$ is a bounded continuous function on $X$, and as a consequence of the continuity of the action of $G$ on $C(K)$, the orbit map 
\[G\ni g\mapsto {^gf\circ i}={^g(f\circ g)}\in \CB(X)\]
is continuous as a mapping from $G$ to the Banach space (moreover, a $C^\ast$-algebra) $\CB(X)$ of all continuous bounded complex functions on $X$ with the supremum norm. 

This allows one to retrieve the {\em maximal} (or: {\em universal}) equivariant compactification of $X$, as follows. Denote $\UCB_G(X)$ the family of all functions $f\in \CB(X)$ for which the orbit map $g\mapsto {^gf}$ is continuous. In other words, one has:
\[\forall \e>0,~~\exists V\ni e,~\forall x\in X~\forall v\in V~~\abs{f(x)-f(vx)}<\e.\]
(Such functions are sometimes called either {\em $\pi$-uniform} or {\em $\alpha$-uniform} on $X$, where $\pi$ or $\alpha$ is the symbol for the action of $G$ on $X$. Another possibility is to simply call those functions {\em right uniformly continuous} on $X$.) This $\UCB_G(X)$ is a $C^\ast$-subalgebra of $\CB(X)$, and so $G$ acts continuously on the corresponding maximal ideal space,  $\alpha_G(X)$. Equipped with the canonical continuous map $X\mapsto \alpha_G(X)$ having a dense image (evaluations at points $x\in X$), this is the universal equivariant compactification of $X$: indeed, for any other compactification $K$, the functions $f\circ i$, $f\in C(X)$, form a $C^\ast$-subalgebra of $\UCB_G(X)$, meaning that $X$ is an equivariant factor-space of $\alpha_G(X)$.

Thus, in order to prove that every equivariant compactification of a $G$-space $X$ is trivial, it is enough to show that $\alpha_G(X)$ is a singleton, or, the same, $\UCB_G(X)$ consists of constant functions only.

All of the above facts are quite easy observations, and the construction appears not only in the context of continuous actions, but also Borel and measurable ones. (Cf. e.g. \cite{GTW,Gl-W3}; we will allude to this parallel again in the Conclusion.)

For a more detailed treatment with numerous related results, references, and open problems, see Megrelishvili's survey paper \cite{megrelishvili07}.

\section{Equivariant metric fan}

Here we revise Megrelishvili's important construction from \cite{megrelishvili,megrelishvili89}.

\subsection{}
Denote $M$ the usual metric fan, that is, the union of countably many copies of the unit interval, $[0,1]\times\N$, with the base $\{0\}\times\N$ glued to a single point, $\ast$, and the topology given by the graph metric. In particular, $M$ is homeomorphic to the union of straight line intervals $[0,e_n]$ in the Hilbert space $\ell^2$ joining zero with the basic vectors $e_n$, $n\in\N_+$, with the norm topology. This is the cone over the discrete space $\N$ of the positive integers, except that the topology is weaker than that in the cone construction as used in algebraic topology.

The main invention of Megrelishvili is the clever choice of an acting group, $G$. It consists of all homeomorphisms of the fan keeping fixed each point located at a distance $1/k$, $k=1,2,3,\ldots$ from the base point $\ast$. (We will refer to such points in the future as {\em marked points.}) Thus, each element of $G$, when restricted to an interval of the form $[1/(k+1),1/k]\times\{n\}$, $k,n\in\N_+$, is an orientation-preserving self-homeomorhism of the interval. This means that $G$, as a group, is the product of a countably infinite family of copies of the group of homeomorphisms of the interval. Equip the group of homeomorphisms of every interval with the usual compact-open topology (that is, the topology of uniform convergence on the interval), and $G$, with the corresponding product topology. Now one can easily verify that this topology is the compact-open topology, or the topology of compact convergence on the fan.

The action of $G$ on $M$ is continuous. 
At every point $x\in M\setminus\{\ast\}$ (including the marked points) this is so because $x$ has a $G$-invariant closed neighbourhood homeomorphic to the unit interval, and locally the action of $G$ is just the action of some group of homeomorphisms with the compact-open topology which is continuous. 
At the base point $\ast$, the action is continuous namely due to all those fixed points converging to it in all directions. If $\e>0$, one can select $k$ with $1/k<\e$, then the union $[0,1/k]\times\N$ is a $G$-invariant neighbourhood of $\ast$, contained in the $\e$-ball around $\ast$.

Now let $f$ be a real-valued $\pi$-uniform function on $M$. One can assume that $f$ takes values in the interval $[0,1]$. Fix an $\e>0$. Since $f$ is continuous at the base point, there is $k$ with $\abs{f(x)-f(\ast)}<\e/2$ whenever $d(\ast,x)\leq 1/k$, that is, if $x$ belongs to one of the intervals $[0,1/k]\times\{n\}$. 

For each one of the fixed points $(1/k,n),(1/(k-1),n),\ldots,(1,n)$, $n\in\N$, because of the continuity of $f$, there is $\e_{i,n}>0$ such that the difference of value of $f$ at the points $\left(\frac 1i\pm \e_{i,n},n\right)$ is less than $\e/2k$. We can assume $\e_{i,n}$ to be so small that the open interval of radius $\e_{i,n}$ around $(1/k,n)$ contains no other marked points.
Since $f$ is $\pi$-uniform, for some neighbourhood $V$ of the identity in $G$ one has $\abs{f(x)-f(vx)}<\e/2k$ for each $x\in M$ and $v\in V$. In view of the definition of the topology of $G$, for all $n$ sufficiently large, $n\geq N$, the restriction of $V$ to every interval between two consequent fixed points along $[0,1]\times \{n\}$ contains all homeomorphisms of this interval. In particular, suitable elements of $V$ take 
$\left(\frac 1k+\e_{k,n},n\right)$ to $\left(\frac 1{k-1}-\e_{k,n},n\right)$, $\left(\frac 1{k-1}+\e_{k,n},n\right)$ to $\left(\frac 1{k-2}-\e_{k,n},n\right)$, and so on. The triangle inequality implies that $\abs{f(1,n)-f(\ast)}<\e$ for all $n\geq N$. 

By the definition of the topology of the universal compactification $\alpha_G(M)$, the images $i(1,n)$, $n\to\infty$ converge to $i(\ast)$. The continuous equivariant mapping $i$ is no not a homeomorphic embedding. 

\subsection{} The above construction was further modified by Megrelishvili as follows. Let $M(x)$ and $M(y)$ be two copies of the metric fan, having base points $x$ and $y$ respectively, where $x\neq y$. The same group $G$ acts continuously on both spaces. Now glue $M(x)$ and $M(y)$ together along the respective endpoints $(1,n)$, $n\in\N$, with the quotient topology (the same: the topology given by the graph metric). Thus, we obtain a ``metric suspension'' over $\N$. Since the endpoints are all fixed, the action of $G$ on the resulting space is well-defined and continuous. In the universal $G$-compactification of the ``double fan'' $M(x,y)$, one has $i(1,n)\to i(x)$ and $i(1,n)\to i(y)$, and so $i(x)=i(y)$. The mapping $i$ is not even injective. 

\subsection{} In this paper we find it more convenient to replace the acting group, $G$, with a product of two copies of the same topological group, $G_x\times G_y$, where $G_x$ acts on $M(x)$ and $G_y$ acts on $M(y)$. 
The action by each of the respective groups, $G_x$ and $G_y$, extend over the disjoint union $M(x)\sqcup M(y)$  by acting on the ``complementary part'' in a trivial way, that is, leaving each point fixed. This way, we obtain a continuous action of the product group $G_x\times G_y$. After we glue $M(x)$ and $M(y)$ together along the endpoints $(1,n)$, $n\in\N$, the action of the product group $G_x\times G_y$ on $M(x,y)$ is still well-defined and it is easy to see that it is continuous. Again, in the universal $G_x\times G_y$ equivariant compactification of $M(x,y)$, the images of the points $x$ and $y$ are identical.

\subsection{} Megrelishvili's double metric fan, or metric suspension, will be the key element in the construction below. However, we need to introduce slight modifications. 

First of all, notice that the topology on the group $G$ can be strengthened to the topology of pointwise convergence on $M$ viewed as discrete. In other words, the basic system of neighbourhoods of identity is formed by finite intersections of stabilizers $\St_{x_1}\cap\ldots\cap \St_{x_n}$, for all finite collections $x_1,x_2,\ldots,x_n\in M$. This is a group topology.
On the group of homeomorphisms of the interval, this topology is easily seen to be finer than the compact-open topology: indeed, given $\e>0$, find $k>\e^{-1}$, and if a homeomorphism stabilizers every point $1/k, 2/k,\ldots, (k-1)/k$, then it moves any point of the interval by less than $\e$. And on the group $G$, the pointwise topology is that of the product of groups of homeomorphisms of the individual intervals between marked points, if each one of them is equipped with the pointwise topology. Consequently, our topology is strictly finer than the topology used by Megrelishvili, and so the action is still continuous. This fact is easily seen directly as well: for points $x$ other than the base point, one can have $x$ ``boxed in'' on both sides by means of two points $y,z$ stabilized by an open neighbourhood, $V\ni e$, and as a result, every element of $V$ will take the interval $(y,z)$ to itself. At the base point, the continuity of the action is, as before, a consequence of the existence of the uniformly convergent sequences of fixed points in all directions. 

When this topology is used, the interval $[0,1]$ can be replaced just with a linearly ordered set, $(X,\leq)$. The only restrictions are: $X$ contains the smallest element to which there converges a strictly decreasing countable sequence, and the order-preserving bijections stabilizing elements of this sequence (of marked points) act transitively on each interval between two consecutive marked points. 
In this paper, we will make the most economical selection, namely, the interval $[0,1]_{\Q}$ of the rational numbers. Also, technically it would be more convenient for us to fix as the convergent sequence $(2^{-k})_{k=0}^{\infty}$ rather than $(1/k)_{k=1}^{\infty}$.

\subsection{}
To sum up, we will denote $M(x,y)$ the {\em rational Megrelishvili double fan}. As a set, it is obtained from $[0,1]_\Q\times \N\times \{x,y\}\times \{(x,y)\}$ by identifying 
\begin{itemize}
\item
        all the points $(0,n,x,(x,y))$, $n\in\N_+$, with the point $x$, 
\item all the points $(0,n,y,(x,y))$, $n\in\N_+$,  with the point $y$, and 
\item every point $(1,n,x,(x,y))$ with the point $(1,n,y,(x,y))$, $n\in\N_+$. 
        \end{itemize}
        
 Notice that the order of $x$ and $y$ matters, which is why we have included the indexing pair $(x,y)$ in the definition. However, if it is clear what pair we are talking about, we will still simply talk of the points $(t,n,x)$ and $(t,n,y)$, suppressing the index $(x,y)$.
        
The topology on the space $M(x,y)$ is given by the graph distance, that is, the length of the shortest path joining two points. A path between two points $a,b$ is a sequence of points $x_0=a, x_1,\ldots,x_n=b$, where any two adjacent points belong to the same edge. Thus, $M(x,y)$ is a countable metric space.
        
A topological realization of $M(x,y)$ can be obtained inside the space $\ell^2(\N)$ by identifying $x$ with $e_0$, $y$ with $-e_0$, and by joining $e_0$ and $-e_0$ with each vector $e_n$, $n\in\N_+$, by means of a straight line segment in the space $\ell^2$ viewed as a vector space {\em over the rationals,} $\Q$. Then induce the norm topology from $\ell^2$.

The group $G=G_{(x,y)}$ of transformations of $M(x,y)$ consists of all homeomorphisms which keep each point of the form
\[\left(2^{-k},n,x\right), ~\left(2^{-k},n,y\right),~k\in\N,~n\in\N_+,\]
fixed. The restriction of the action on each interval with the endpoints $(2^{-k-1},n,a)$ and $(2^{-k},n,a)$, $a\in \{x,y\}$, consists of all order-preserving bijections of the interval.
The topology on $G$ is that of a simple convergence on $M(x,y)$ viewed as discrete, that is, induced by the embedding of $G$ into the space $M(x,y)$ with the discrete topology raised to the power $M(x,y)$.

Since the space is countable, the group $G$ is separable metrizable, and in fact Polish. Indeed, one can easily see that $G$ is the product of countably many copies of the well-known Polish group $\Aut(\Q,\leq)$ of all order-preserving bijections of the rationals, with its standard topology of simple convergence on $\Q$ as discrete.
The group $G$ acts continuously on $M(x,y)$, and for every equivariant compactification of this space, one has $i(x)=i(y)$, with the same exact argument as in the original Megrelishvili's example.

\subsection{}
The rational version of Megrelishvili's metric fan is not new. It was first described by J. van Mill (\cite{vM}, Remark 3.4), see also \cite{kozlov}, example 4 in \S 5.

\section{Attaching a double fan to every pair of points of a $G$-space}

Given a topological $G$-space $X$, the space that we will denote $M(X)$ is obtained from $X$ by attaching to every ordered pair of distinct points $x,y\in X$ a double rational fan $M(x,y)$, and subsequently enlarging the transformation group so as to bring in all the copies of $G_{(x,y)}$ acting on each $M(x,y)$.
Notice that we attach two different fans to each pair $\{x,y\}$, one corresponding to the ordered pair $(x,y)$, and the other, to the pair $(y,x)$, this will be necessary later on for the symmetry reasons.
Under some restrictions on $(X,G)$, the resulting action is continuous.
As a consequence, within every equivariant compactification of the resulting space $M(X)$, the original space $X$ will collapse to a point. Here is a more precise description.

\subsection{}
One of the oldest results of uniform topology (see e.g. a preliminary draft \cite{bourbaki} of the corresponding chapter of Bourbaki's {\em Topologie G\'en\'erale}) states that a topological space is Tychonoff if and only if the topology is generated by a separated uniform structure, or, equivalently, by continuous pseudometrics separating points.

Let $X$ be a Tychonoff topological space. As usual, $X^2\setminus\Delta$ denotes the collection of all ordered pairs of different points of $X$.  

Consider the disjoint sum $\sqcup_{(x,y)\in X^2\setminus\Delta} M(x,y) $ of copies of the double rational fan, and glue this sum to $X$ along the canonical map sending each fan endpoint to the corresponding point of $X$. Denote the resulting set $M(X)$.

The first task is to topologize $M(X)$. Let $d$ is a pseudometric on $X$. Denote $\bar d$ the corresponding path pseudometric on $M(X)$, that is, $\bar d(a,b)$ is the smallest among the numbers $\sum_{i=0}^n d(x_i,x_{i+1})$, where $x_0=a$, $x_n=b$, and the distance between every two consecutive points $x_i,x_{i+1}$ is calculated either to $X$, or in some $M(x,y)$. 

In general, the path length can be arbitrarily long, but
if $d\leq 2$ (which is not a restrictive condition from the view of topology that $d$ generates), then $n\leq 3$ suffices. For instance, in this case
\[\bar d(x,y) = d(x,y)\mbox{ for }x,y\in X,\] 
\[\bar d((t,n,x,(x,y)),z) = \min\{t + d(x,z),2-t+d(y,z)\}\mbox{ if }x,y,z\in X,\]
and so forth. Also, one has $\bar d\leq 4$.

We will topologize $M(X)$ with the set of all pseudometrics of the form $\bar d$, as $d$ runs over all continuous pseudometrics on $X$ with $d\leq 2$. The following should be now obvious.

\begin{lemma}
\label{l:4.1}
        The space $M(X)$ is a Tychonoff space containing $X$ and each fan $M(x,y)$ as closed  subspaces. Every set $M(x,y)\setminus\{x,y\}$ is open in $M(X)$. 
        \qed
\end{lemma}
 
The resulting topology does not depend on the choice of a family of pseudometrics $d$ once it generates the topology of $X$. 

\begin{lemma}
        Let $D$ be some family of pseudometrics generating the topology of $X$ and satisfying $d\leq 2$ for all $d\in D$. The pseudometrics $\bar d$, $d\in D$ generate the same topology on $M(X)$ as above.
        \label{l:D}
\end{lemma}

\begin{proof}
        We need to verify that the same topology as defined by us previously is generated at every point of $M(x,y)$. For the interior points of each fan, this is clear. Let therefore $x\in X$ and let $V$ be a neighbourhood of $x$ in $M(X)$. There is a continuous pseudometric $\rho$ on $X$ and $\e$, $0<\e<1$, so that the corresponding open ball around $x$, formed in $M(X)$, is contained in $V$,
        \[V\supseteq B(x,\e,{\bar \rho}).\]
        By our assumption on $D$, there is $d\in D$ and a $\delta>0$ with the property that $B(x,\delta,d)\subseteq B(x,\e/2,\rho)$ (the balls in $X$). Also, without loss in generality, one can assume that $\delta<\e/2$. If now $z\in B(x,\delta,\bar d)$ (the ball formed in $M(X)$), then for some unique $y,w\in X$ one has either $z=(t,n,y,(y,w))$ or $z=(t,n,y,(w,y))$ and
        \[\bar d(x,z) \leq t + d(x,y) < \delta<\frac{\e}2,\]
        and
        \begin{eqnarray*}
                \bar\rho(x,z) &\leq & t + \rho(x,y) \\
                &<& \frac{\e}2 + \frac{\e}2 = \e.
                \end{eqnarray*}
\end{proof}

\subsection{}
Next we will extend the action of the group $G$ from $X$ over $M(X)$. 

Let first $g$ be an arbitrary self-homeomorphism of $X$. Then $g$ extends to a homeomorphism of $X^2\setminus\Delta$, as $(x,y)\mapsto (gx,gy)$. Given any $(x,y)$ and a point $(t,n,a)\in M(x,y)$, $a\in\{x,y\}$, define 
\[g(t,n,a,(x,y)) = (t,n,ga,(gx,gy))\in M(gx,gy).\]
It is clear that this defines a homeomorphism of $M(x,y)$ onto $M(gx,gy)$. 
In particular, if $g$ swaps $x$ and $y$, then $M(x,y)$ is being homeomorphically mapped onto $M(y,x)$. In particular, the extension of $g$ over $M(X)$ is bijective.

\begin{lemma}
        Every extension of a homeomorphism $g$ of $X$ as above is a homeomorphism of $M(X)$. Also, the extended homeomorphism preserves the family of the marked points of all the double fans, $(2^{-k},n,a)$.
        \label{l:extensionofg}
\end{lemma}

\begin{proof}
        Again, it is quite obvious that $g$ is a local homeomorphism (in fact, a local isometry) at the interior point of every double fan. It is also clear that the extension of $g^{-1}$ is the inverse to the extension of $g$.
        It remains to verify that $g$ is continuous at an arbitrary point $x\in X$. Given a neighbourhood $V$ of $gx$, find a continuous pseudometric $d$ on $X$ and an $\e>0$ with $B(gx,\e,\bar d)\subseteq V$. Define a continuous pseudometric $\rho$ on $X$ by $\rho(a,b) = d(ga,gb)$. Now it is easy to see that the image of the ball $B(x,\e,\bar\rho)$ under $g$ is contained in (in fact, coincides with) $B(gx,\e,\bar d)$. 
        The second statement is clear.
\end{proof}

In this way, the group $G$ acts on $M(X)$ by homeomorphisms.

\begin{lemma}
        Assume the topology of the group $G$ is that of simple convergence on $X$ as discrete. Then it is also the topology of simple convergence on $M(X)$ as discrete. If in addition $G$ acts on $X$ continuously, then $G$ acts on $M(X)$ continuously as well.
\end{lemma}

\begin{proof}
        Let $\tau$ denote the original topology on $G$ (that is, the topology of simple convergence on $X$ viewed as a discrete space), and $\varsigma$, the topology of simple convergence on $M(X)$ equipped with the discrete topology. Then $\tau\subseteq\varsigma$. For every point of the form $(t,n,a)\in M(x,y)$, with $a\in\{x,y\}$, its stabilizer in $G$ consists exactly of all $g$ that leave both $x$ and $y$ fixed. 
        This implies $\varsigma\subseteq\tau$. 
        
        The continuity of the action of $G$ on $M(X)$ needs to be verified separately at the interior points of the fans and at the points of $X$. Let $x,y\in X$, $(t,n,a)\in M(x,y)$ with $0<t\leq 1$ and $a\in\{x,y\}$. Take an arbitrary neighbourhood of this point, $W$, and set $U=W\cap M(x,y)\setminus\{x,y\}$.
        The set $\St_x\cap\St_y$ is an open neighbourhood of the identity element in $G$ and it sends the open neighbourhood $U$ of $(t,n,a)$ to itself.
        
        Now let $x\in X$. It is enough to look at a neighbourhood of $x$ of the form $B(x,\e,\bar d)$, where $d$ is a continuous pseudometric on $X$ and $\e<1$. Since the action of $G$ on $X$ is continuous, there is a continuous pseudometric $\rho$ and $\delta>0$, as well as a neighbourhood $V$ of the identity in $G$, satisfying $vy\in B(x,\e/2,d)$ once $v\in V$, $y\in X$, and $\rho(x,y)<\delta$. We can further assume that $\delta\leq \e/2$ and $\delta<1$. Now let $z\in B(x,\delta,\bar\rho)$, where the ball is formed in $M(X)$. Then $z$ is of the form $(t,n,y,(a,b))$, where $y\in\{a,b\}$, $t\in [0,1]$, and $\rho(x,y)+t<\delta$. In particular, $t<\delta\leq \e/2$ and $\rho(x,y)<\delta$. Let $v\in V$. Then $vz = (t,n,vy,(va,vb))$, and we conclude:
        \[\bar d(vz,x)\leq t + d(vy,x) <\e.\] 
\end{proof}

\subsection{} 
Suppose a topological group $G$ act continuously on a topological space $X$ in such a way that the topology of $G$ is the topology of pointwise convergence on $X$ viewed as discrete. We have seen that the action of $G$ extends over $M(X)$, and the topology of $G$ is that of simple convergence on $M(X)$ with the discrete topology.
Next we will add to the acting group $G$ all the topological groups $G_{(x,y)}$ acting on each double fan and keeping the rest of the space $M(X)$ fixed. 

Of course if two different topological groups continuously act on the same topological space, it is not in general true that they can be jointly embedded in a larger topological group which still acts continuously. Such examples can be found even if one of the two acting groups is discrete. Therefore, in our case, we need to construct this larger group explicitely. It is going to be a generalized wreath product of $G$ with $G_{(\ast,\star)}$, where $\ast,\star$ are two distinct abstract points.

Consider the product group $G_{(\ast,\star)}^{X^2\setminus\Delta}$ with the product topology. We will canonically identify it with the product group $\prod_{(x,y)\in X^2\setminus\Delta}G_{(x,y)}$. For every $h\in G_{(\ast,\star)}^{X^2\setminus\Delta}$, $h=(h_{(x,y)})_{(x,y)\in X^2\setminus\Delta}$, define a self-bijection of $M(X)$ as follows: if $z\in M(x,y)$, $z=(t,n,a)$, $a\in\{x,y\}$, then  
\[h(z) = \left(h_{(x,y)}(t,n,a) \right)\in M(x,y).\]

In this way, we obtain an action of the group $G_{(\ast,\star)}^{X^2\setminus\Delta}$ on $M(X)$. Our aim is to show that it is a continuous action by homeomorphisms, and the topology of the group is that of pointwise convergence on $M(X)$ viewed as discrete. We start with a technical observation.

\begin{lemma} Let $z\in X$ and $0<\e<1$, and let $d$ be a continuous pseudometric on $X$. Then the image of the ball $B(z,\e/2,\bar d)$ formed in $M(X)$ under every transformation $h\in G_{(\ast,\star)}^{X^2\setminus\Delta}$ is contained in $B(z,\e,\bar d)$. 
        \label{l:sepres}
\end{lemma} 

\begin{proof}
        Every point $x\in X$ is fixed by our group, as are the points of each convergent sequence $(2^{-k},n,x)$, $k\in\N$, inside of $M(a,b)$, where $x\in\{a,b\}$. The intersections of $V=B(z,\e,\bar d)$ and of $U=B(z,\e/2,\bar d)$ with the interval $[0,1]\times \{(n,x)\}$, provided they are non-empty, are the semi-open intervals $[0, \e-d(z,x)) \times \{(n,x)\}$ and $[0, \e/2-d(z,x)) \times \{(n,x)\}$ respectively, and there is at least one marked point of the form $(2^{-k},n,x)$ between the endpoints of the two intervals. Since the marked point is fixed by every $h\in G_{(\ast,\star)}^{X^2\setminus\Delta}$, the conclusion follows.
\end{proof}

\begin{lemma} The above action of the group $G_{(\ast,\star)}^{X^2\setminus\Delta}$ on $M(X)$ is a continuous action on $M(X)$ by homeomorphisms. In addition, the product topology on this group is the topology of simple convergence on $M(X)$ viewed as discrete.
\end{lemma} 

\begin{proof} 
        To verify that every $h$ gives a homeomorphism of $M(X)$, it is enough to check that $h$ is continuous at every point. For the interior points of the fans, it is clear from the definition, and for $x\in X$, it follows from Lemma \ref{l:sepres}.
        
        Let us verify the continuity of the action at an arbitrary point $z\in M(X)$. For $z\in X$, this is again a consequence of Lemma \ref{l:sepres}. 
        Now let $z$ be an interior point of some double fan, $M(x,y)$. Let $V$ be a neighbourhood of $z$. There is a neighbourhood $O$ of the identity in the group $G_{(x,y)}$ and an open neighbourhood $U$ of $z$ in $M(x,y)$ with $O\cdot U\subseteq V$. Denote $\tilde O$ the standard basic neighbourhood of the identity in the product group $G_{(\ast,\star)}^{X^2\setminus\Delta}$ which is a cylinder set over $O$, that is, $h\in\tilde O$ if and only if $h_{(x,y)}\in O$. Then $\tilde O\cdot U\subseteq V$.
        
        To verify the last statement of Lemma,
        if $F$ is a finite subset of $M(X)$, then for every
        pair $(x,y)\in X^2\setminus\Delta$ the stabilizer of $F\cap M(x,y)$ in $G_{(x,y)}$ is an open subgroup, and all but finitely many such stabilizers coincide with the entire group $G_{(x,y)}$. It follows that the pointwise stabilizer of $F$ is a standard basic neighbourhood of the identity element in the product. Since the topology of each $G_{(x,y)}$ is that of simple convergence on $M(x,y)$ with the discrete topology, we conclude that the pointwise stabilizers of finite subsets of $M(X)$ form a basic neighbourhood system for the product topology. 
\end{proof}

\subsection{}
The group $G$ acts on $X$, and hence on $X^2\setminus\Delta$, and under our assumptions, the topology of $G$ is that of simple convergence on $X^2\setminus\Delta$ viewed as discrete. 
Further, $G$ acts on 
$G_{(\ast,\star)}^{X^2\setminus\Delta}$ by coordinate permutations: if $g\in G$, $h\in G_{(\ast,\star)}^{X^2\setminus\Delta}$, and $(x,y)\in X^2\setminus\Delta$, then
\[({^g h})_{(x,y)} = h_{(g^{-1}x,g^{-1}y)}.\]
This is an action of $G$ by automorphisms of the topological group $G_{(\ast,\star)}^{X^2\setminus\Delta}$.

[ Recall that if a topological group $G$ acts by automorphisms on a topological group $H$, and the action is continuous, then the semidirect product $G\ltimes H$ is the cartesian product $G\times H$ with the product topology and the group operation
        \[(a,b)(c,d)=(ac,b\cdot{^ad})\]
is a topological group. It contains both $G$ and $H$ as closed topological subgroups in a canonical way, and $H$ is normal, with
\[ghg^{-1} = {^gh}.~~]\]

\begin{lemma} 
Under our assumptions on $G$, the above action is continuous, and so the corresponding semidirect product $G\ltimes G_{(\ast,\star)}^{X^2\setminus\Delta}$ is a topological group.
\end{lemma}

\begin{proof}
        Let $h\in G_{(\ast,\star)}^{X^2\setminus\Delta}$ and let $V$ be a neighbourhood of $h$. It is enough to consider in place of $V$ a basic product neighbourhood: $v\in V$ if and only if $v_x\in V_x$ for all $x$ belonging to a finite set $F\subseteq X$, where each $V_x$ is a neighbourhood of $h_x$ in $G_{(\ast,\star)}$. The open subgroup $O=\cap_{x\in F}\St_x$ of $G$ has the property $O\cdot V=V$.
\end{proof}

\begin{lemma}
        The topological group $\tilde G = G\ltimes G_{(\ast,\star)}^{X^2\setminus\Delta}$ acts continuously on $M(X)$ extending the actions of $G$ and of $G_{(\ast,\star)}^{X^2\setminus\Delta}$. The topology of $\tilde G$ is that of pointwise convergence on $M(X)$ viewed as discrete.
\end{lemma}

\begin{proof}
        To verify that the actions of $G$ and of $G_{(\ast,\star)}^{X^2\setminus\Delta}$ combine together to give an action of the semidirect product, it is enough to verify that for each $g\in G$ and $h\in G_{(\ast,\star)}^{X^2\setminus\Delta}$,
        the action by the element ${^g h}$ on $M(X)$ equals the composition of  actions of three elements, $g\circ h\circ g^{-1}$. For every pair $(x,y)$ and each point $(t,n,a,(x,y))$, $a\in\{x,y\}$, $t\in [0,1]$, one has
        \begin{eqnarray*}
                g\circ h\circ g^{-1}(t,n,a,(x,y)) &=& g\circ h (t,n,g^{-1}(a),(g^{-1}x,g^{-1}y)) \\
                &=& g\left(h_{(g^{-1}x,g^{-1}y)}(t,n, g^{-1}(a),(g^{-1}x,g^{-1}y))\right) \\
                &=& h_{(g^{-1}x,g^{-1}y)}(t,n,a,(x,y))) \\
                &=& {^g h}(t,n,a,(x,y)).
                \end{eqnarray*}
        Every element $(g,h)$ of $\tilde G$ can be uniquely written as the product $(e,h)(g,e)$, where $g\in G$, $h\in G_{(\ast,\star)}^{X^2\setminus\Delta}$, and the rule 
        \[(g,h)(x) = h(g(x))
        \]
        consistently defines an action. Indeed, 
        \begin{eqnarray*}
                (a,b)(c,d)(x)&=& badc(x) \\
                &=& bada^{-1}ac(x) \\
                &=& (e,b)(e,{^ad})(ac,e)(x) \\
                &=& (e,b{^ad})(ac,e)(x) \\
                &=& (ac,b\,{^ad})(x).
                \end{eqnarray*}
        
        Since the actions of both $G$ and $G_{(\ast,\star)}^{X^2\setminus\Delta}$ are continuous, it follows immediately that so is the action of the semidirect product. 
        
        The topologies of both groups are those of simple convergence on $M(X)$ viewed as discrete. 
        Let $F$ be a finite subset of $M(X)$. Denote $\Phi$ the union of all sets $\{x,y\}$ having some $z\in F$ belonging to $M(x,y)$. Then the stabilizer of $F$ in $\tilde G$ is the cartesian product of the stabilizer of $\Phi$ in $G$ and a standard open neighbourhood of the identity in the product group, $U= \prod_{x\in X} U_x$, where each $U_x$ is a neighbourhood of the identity in the group $G_{(a_x,b_x)}$, stabilizing a finite set, and $U_x$ equals the entire group whenever $x\notin \Phi$. These observations implies the last statement of Lemma. 
\end{proof}

\section{Iterating the procedure}

\subsection{}
Let $X$ be a $G$-space such that the topology of $G$ is that of simple convergence on $X$ as discrete. 
Repeat iteratively the extension ($M(X)$, $\tilde G$) countably infinitely many times. Denote the corresponding $n$-th iterations by $M^n(X)$ and $\tilde G^{(n)}$.

Each space $M^n(X)$ is contained inside $M^{n+1}(X)$, so we can form the union $M^{\infty}(X)$. Let us equip it with a topology. If $d$ is a continuous pseudometric on $X$ with the property $d\leq 2$, the graph metric extension $\bar d$ over $M(X)$ satisfies $\bar d\leq 4$. The pseudometric $\tilde d = \min\{\bar d, 2\}$ is continuous and induces the same topology as $\bar d$. Getting back to our construction, denote the $n$-th iteration of $d$ corresponding to the extension $\tilde d$ by $\tilde d^{(n)}$. This is a pseudometric on $M^{\infty}(X)$, and we equip the space with the topology generated by all such pseudometrics. Notice that the restriction of $\tilde d^{(\infty)}$ to each $M^{(n)}(X)$ equals $\tilde d^{(n)}$.

The action of each group $\tilde G^{(n)}$ on $M^n(X)$ lifts recursively to an action of $\tilde G^{(n)}$ on $M^{\infty}(X)$. Notice that each subspace $M^{(m)}(X)$ is invariant under the action of every group $\tilde G^{(n)}$ for all $m,n$. 

\subsection{}
From now on, we will be only working with a particular case of interest to us. Namely, we will start with the two-point space $X=\{\ast,\star\}$ such that $d(\ast,\star)=2$, equipped with an action of a trivial group $G=\{e\}$.
The resulting space $M^{\infty}(X)$ is a countable metric space with the metric $\tilde d^{\infty}$, a graph with rational edges, infinitely branching at every point, and equipped with the graph distance. Because of Lemma \ref{l:D}, the topology induced by the metric $\tilde d^{\infty}$ on every subspace $M^{\infty}(X)$ is the same as the topology we have defined elsewhere.

\begin{lemma}
        Let $n\in\N$ and $g\in \tilde G^{(n)}$. The recursive extension of $g$ over $M^{\infty}(X)$ is a homeomorphism.
\end{lemma}

\begin{proof}
        We have noted before that the extension of $g^{-1}$ from $X$ over $M(X)$ is the inverse of the extension of $g$ for every topological space $X$, therefore the same holds for the extension over $M^{\infty}(X)$. It is enough to verify that $g$ is continuous at an arbitrary point $z\in M^{\infty}(X)$. 
        
        Suppose first $z\notin M^n(X)$. Then $\e=\tilde d^{(\infty)}(z, M^n(X))>0$, and the restriction of $g$ to the open $\e$-ball around $z$ in $M^{\infty}(X)$ is an isometry, as seen from the recursive way in which $g$ is being extended.
        
        Now suppose $z\in M^n(X)$. Given a neighbourhood $V$ of $gz$, find an $\e$ with $0<\e<1$ and $B(gz,\e,\tilde d^{(\infty)})\subseteq V$, where the open ball is formed in the metric space $M^{\infty}(X)$. There is $\delta$ satisfying $0<\delta<\e/2$ and $gB(z,\delta,\tilde d^{(n)})\subseteq B(gz,\e/2,\tilde d^{(n)})$ (the balls in $M^n(X)$). If $x\in B(z,\delta,\tilde d^{(\infty)})$, then there are $a,b\in M^n(X)$ and $w\in\{a,b\}$ with $\tilde d^{(\infty)}(x,z) = \tilde d^{(\infty)}(x,w) + \tilde d^{(n)}(w,z)$, and we have
        \[\tilde d^{(\infty)}(gx,gz) \leq \tilde d^{(\infty)}(gx,gw) + \tilde d^{(n)}(gw,gz) < \tilde d^{(\infty)}(x,w) +\e/2 < \e.\] 
\end{proof}

Thus, the group $\tilde G^{(\infty)}=\cup_{n\in\N} \tilde G^{(n)}$ acts on the countable metric space $M^{\infty}(X)$ by homeomorphisms. We equip it with the topology of simple convergence on $M^{\infty}(X)$ viewed as discrete. This is a separable, metric group topology, inducing the usual topology on every subgroup $\tilde G^{(n)}$.

\begin{lemma}
        The action of the topological group $\tilde G^{(\infty)}$ on the space $M^{\infty}(X)$ is continuous.
\end{lemma}

\begin{proof} We will verify the continuity of the action at an arbitrary point $z\in M^{\infty}(X)$. Let $0<\e<1$. Denote $n=\min\{m\colon z\in M^{(m)}(X)\}$. 
        Find a finite set $F\subseteq M^{(n)}(X)$ and a $\delta$ with $0<\delta<\e/4$, so that one has $gB(z,\delta,\tilde d^{(n)})\subseteq B(z,\e/2, \tilde d^{(n)})$ as soon as $g\in \tilde G^{(n)}$ and $g$ stabilizes each point of $F$. 
        
        We claim that $g^\prime B(z,\delta,\tilde d^{(\infty)})\subseteq B(z,\e/2, \tilde d^{(\infty)})$ as soon as $g^\prime\in \tilde G^{(\infty)}$ and $g^\prime$ stabilizes each point of $F$. Indeed, such a $g^{\prime}$ can be written $g^{\prime}=gh$, where $g\in \tilde G^{(n)}$, $g$ stabilizes each point of $F$, and $h$ stabilizes each point of $M^{(n)}(X)$. Let $x\in B(z,\delta,\tilde d^{(\infty)})$. Then for some $m\geq n$, 
        \[\tilde d^{(\infty)}(x,z) = \sum_{i=n}^m \abs{x_i-x_{i+1}} + \tilde d^{(n)}(x_n,z),
        \]
        where $x_i\in M^{(i)}(X)$, and the notation is slightly abused to stress the fact that the distance between each two consecutive points is taken within a suitable rational interval. In particular, $\sum_{i=n}^m \abs{x_i-x_{i+1}}<\delta$. For every $i$, there is a marked point along the edge starting at $x_i$ and going towards $x_{i+1}$ and beyond, at some distance $2^{-k_i}$ from $x_i$ so that $\abs{x_i-x_{i+1}}\leq 2^{-k_i}< 2\abs{x_i-x_{i+1}}$. This implies that the distance between $x_i$ and $x_{i+1}$ cannot be increased by more than twice under any transformation, and we have
        \begin{eqnarray*}
                \tilde d^{(\infty)}(g^{\prime}x,g^{\prime}z) + \tilde d^{(n)}(gx_n,gz)&\leq &
                \sum_{i=n}^m \abs{g^{\prime}x_i-g^{\prime}x_{i+1}}\leq \sum_{i=n}^m 2^{-k_i} + d^{(n)}(gx_n,gz) \\ &<& 2\delta + \frac{\e}2 < \e.
        \end{eqnarray*}
\end{proof}

\subsection{}
Consider the universal equivariant compactification of $M^{\infty}(X)$ under the action of $\tilde G^{(\infty)}$. For any $x,y\in M^{\infty}(X)$, there is $n$ with $x,y\in M^{(n)}(X)$. Therefore, $x$ and $y$ are joined by an equivariant Megrelishvili double fan $M(x,y)$ inside $M^{(n+1)}(X)$. The topological group $G_{x,y}$ is a topological subgroup of $\tilde G^{(n+1)}$, therefore of $\tilde G^{(\infty)}$, with the standard action on $M(x,y)$. Under the mapping $i$, one has $i(x)=i(y)$. We conclude: the entire universal compactification of $M^{\infty}(X)$ is a singleton.

\subsection{}
Finally, notice that, according to a well-known result of Sierpinski \cite{sierpinski}, a countable metrizable space without isolated points is homeomorphic to the space of rational numbers, $\Q$, with the usual topology. Thus, topologically, $M^{\infty}(X)$ is just $\Q$.

\section{Concluding remarks}

\subsection{} The referee suggested that the recent preprint \cite{GJ} constructing universal compact metrizable $\R$-spaces can be used to give an explicit proof of de Vries' theorem \cite{dV-1,dV} in the particular case $G=\R$. 

\subsection{}
Both $X$ and $G$ in our example are separable metrizable.
The author does not know if the example can be modified so as to have both $X$ and $G$ Polish (separable completely metrizable). 

\subsection{}
The example constructed here is of course artificial. However, if the story of extreme amenability is any indication (see e.g. \cite{P06}), the phenomenon may one day be naturally found. One needs also compare the phenomenon discovered and explored by Glasner, Tsirelson and Weiss \cite{GTW} and Glasner and Weiss \cite{Gl-W3} whereby a weakly continuous action of a Polish group on a standard probability measure space cannot be realized spatially, that is, as a set-theoretic action on a compact space equipped with an invariant measure. Here, the phenomenon occurs very naturally: every L\'evy group \cite{GrM} of measure-preserving transformations behaves in such a way. 

This reinforces a feeling that there might exist numerous {\em natural} examples of topological transformation groups whose universal equivariant compactification is a singleton.

\subsection{}
For instance, can such an example be realized as the group of automorphisms of a suitable (discrete or continuous) ultrahomogeneous Fra\"\i ss\'e structure? 

The {\em Gromov compactification} of a bounded metric space $X$ corresponds to the $C^\ast$-algebra generated by all distance functions $x\mapsto d(x,-)$; under the action of the isometry group  with the pointwise topology, such functions are always $\pi$-uniform, and so the Gromov compactification is equivariant, and it is a homeomorphic embedding (see e.g. \cite{megrelishvili07}, sect. 2). This means that the topology on the conjectural Fra\'\i ss\'e structure needs to be modified, but again, in some ``natural'' way (given by a suitable partial order for instance, as it is esentially the case in our example).

\subsection{} An elegant ``natural'' example where the universal equivariant compactification has been calculated explicitely is that of the unit sphere $\s^{\infty}$ in the Hilbert space $\ell^2$ under the action of the unitary group $U(\ell^2)$ with the strong operator topology. This compactification is the unit ball ${\mathbb{B}}^{\infty}$ with the weak topology. This is a result by Stojanov \cite{stoyanov}.

This motivated Megrelishvili to ask whether the same conclusion holds for the unit sphere in every separable reflexive Banach space $E$ under the action of the  group of isometries of $E$ with the strong operator topology, see Question 2.5 (2005) in \cite{megrelishvili07}. However, it appears to us the question needs to be adjusted.

According to Jarosz \cite{jarosz}, every Banach space can be renormed in such a way that the group of isometries consists of constant multiples of the identity. For such a renormed space, the Stone-\v Cech compactification of the projective space is an equivariant compactification of the sphere, because the action of the group of isometries is trivial. Since for $E$ infinite dimensional the Stone-\v Cech compactification is non-metrizable, the answer to the question as stated is in the negative. 

However, it makes sense to reformulate the question for separable reflexive Banach spaces whose group of isometries has a dense orbit in the unit sphere. For instance, it appears the answer is already unknown for $L^p(0,1)$, $1<p<\infty$, $p\neq 2$. 

\subsection{} Even for non-reflexive spaces the question makes sense. 
In particular, the group of isometries of the Gurarij space has attacted plenty of interest recently \cite{BY}. Is it true that the universal equivariant compactification of the unit sphere in the Gurarij space under the action of the group of isometries with the pointwise topology is the Gromov compactification of the sphere? 

The same question, for the Holmes space \cite{holmes} (see also \cite{P06}, pp. 112--113).

Notice in this connection that the closed ball of $\ell^2$ in the weak topology is the Gromov compactification of the unit sphere.

\subsection{}
The analogous question, for the Urysohn sphere $\s_{\mathbb U}$ \cite{NVTS}. More precisely: is the universal equivariant compactification of the Urysohn sphere under the action of the group of isometries with the pointwise topology equal to the Gromov compactification of $\s_{\mathbb U}$? Probably the answer is in the positive and should not be very difficult to obtain, as the group of isometries of the Urysohn sphere is by now pretty well understood. 

\subsection{} In conclusion, here is a question suggested by Furstenberg and Scarr (see \cite{megrelishvili07}, question 2.6 (2006)): does there exist a topological transformation group $(G,X)$ whose universal equivariant compactification is a singleton, yet the action of $G$ on $X$ is transitive? 

Such an example cannot be Polish because of Effros' Microtransivitity Theorem combined with the observation that for a closed subgroup $H$ of a topological group $G$, the equivariant compactification $\alpha_G(G/H)$ is always a homeomorphic embedding of $G/H$ \cite{dV-2}.

As noted by Jan van Mill (\cite{vM}, question 3.5) the Furstenberg--Scarr question remains open already for groups acting on the space of rational numbers, $\Q$. Of course in our example the action on $\Q$ is highly non-transitive in view of all those marked points. 

Can there be a suitable group of homeomorphisms of the Hilbert space $\ell^2$ whose equivariant compactification is trivial?

\section{Acknowledgements}
       
I am very indebted to Michael Megrelishvili from whom over the years I have learned many things, including Smirnov's question. Dana Barto\v sov\'a and Micheal Pawliuk were considering my invitation to join as collaborators at the early stages of this project, and the ensuing discussions with them helped to have some of my initial ideas discarded. I am grateful to Konstantin Kozlov for providing some references, and for the anonymous referee for a number of comments.

Much of this work has been done while the author was a 2012--2015 Special Visiting Researcher of the program Science Without Borders of CAPES (Brazil), processo 085/2012.

\end{document}